\documentclass[11pt]{amsart}
\usepackage{geometry}             
\usepackage{graphicx}
\usepackage{amssymb}
\usepackage{amsmath, amsthm}
\usepackage{epstopdf}
\usepackage{tikz}
\usetikzlibrary{backgrounds}
\DeclareMathOperator{\st}{st}
\DeclareMathOperator{\dl}{dl}
\DeclareMathOperator{\lk}{lk}
\DeclareMathOperator{\Sd}{Sd}

\newtheorem{thm}{Theorem}[section]

\newtheorem{prop}{Proposition}[section]
\newtheorem{lem}{Lemma}[section]
\newtheorem{defn}{Definition}[section]

\title{On the Topology of Weakly and Strongly Separated Set Complexes}
\author{Daniel Hess, Benjamin Hirsch}
\address{School of Mathematics, University of Minnesota, Minneapolis, MN 55455}
\email{hessx144@umn.edu}
\address{Department of Mathematics, Harvard University, Cambridge, MA 02138}
\email{bhirsch@college.harvard.edu}

\begin{document}
\maketitle

\begin{abstract} 
We examine the topology of the clique complexes of the graphs of weakly and strongly separated subsets of the set $[n]=\{1,2,\dots,n\}$, which, after deleting all cone points, we denote by $\hat{\Delta}_{ws}(n)$ and $\hat{\Delta}_{ss}(n)$, respectively. In particular, we find that $\hat{\Delta}_{ws}(n)$ is contractible for $n\geq4$, while $\hat{\Delta}_{ss}(n)$ is homotopy equivalent to a sphere of dimension $n-3$. We also show that our homotopy equivalences are equivariant with respect to the group generated by two particular symmetries of $\hat{\Delta}_{ws}(n)$ and $\hat{\Delta}_{ss}(n)$: one induced by the set complementation action on subsets of $[n]$ and another induced by the action on subsets of $[n]$ which replaces each $k\in[n]$ by $n+1-k$.
\end{abstract}

\section{Introduction}

In \cite{Leclerc}, Leclerc and Zelevinsky define the relations of \emph{strong separation} and $\emph{weak separation}$ on the subsets of $[n]=\{1,2,\dots,n\}$. For $A$ and $B$ disjoint subsets of $[n]$, we say that $A$ \emph{lies entirely to the left of} $B$, written $A\prec B$, if $\max(A)<\min(B)$. We say that $A$ \emph{surrounds} $B$ if $A$ can be partitioned into a disjoint union $A=A_1\sqcup A_2$, where $A_1\prec B\prec A_2$.

\begin{defn}
We say that subsets $A,B\subset[n]$ are \emph{strongly separated} from one another if either $A\smallsetminus B\prec B\smallsetminus A$ or $B\smallsetminus A\prec A\smallsetminus B$. 
\end{defn}

\begin{defn}
We say that subsets $A,B\subset[n]$ are \emph{weakly separated} from one another if at least one of the following two conditions holds:\begin{itemize}
\item $|A|\leq |B|$ and $A\smallsetminus B$ surrounds $B\smallsetminus A$
\item $|B|\leq |A|$ and $B\smallsetminus A$ surrounds $A\smallsetminus B$
\end{itemize}
\end{defn}

For each of these relations, we may construct a graph whose vertices are the subsets of $[n]$ and a simplicial complex which is the clique complex of this graph. After removing \emph{frozen} vertices, meaning those corresponding to sets that are strongly or weakly separated from every subset of $[n]$, we denote what remains by $\hat{\Delta}_{ss}(n)$ and $\hat{\Delta}_{ws}(n)$ for the strongly and weakly separated complexes, respectively. In either case, the frozen vertices correspond to initial or final segments of $[n]$, of the form $\{1,2, \dots, k\}$ or $\{k, k+1, \dots, n\}$.

For example, if $n=1$ or $n=2$, then $\hat{\Delta}_{ss}(n)=\hat{\Delta}_{ws}(n)=\varnothing$. The case $n=3$ is pictured in Figure 1 and the case $n=4$ is pictured in Figure 2.

\begin{figure}
\begin{tikzpicture}
[inner sep=0pt, minimum size=7mm]

\node (2) at (-1,0) [circle, draw] {2};
\node (13)	 at (1,0) [circle, draw] {13};
\end{tikzpicture}
\caption{The simplicial complex $\hat{\Delta}_{ss}(3)=\hat{\Delta}_{ws}(3)$.}
\end{figure}
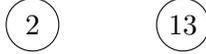

\begin{figure}
\begin{tikzpicture}
[inner sep=0pt, minimum size=7mm]

\node (2) at (-2, 2) [circle, draw, fill=white] {2};
\node (3) at (2,2) [circle, draw, fill=white] {3};
\node (23) at (0,1) [circle, draw, fill=white] {23};
\node (24) at (-1,0) [circle, draw, fill=white] {24};
\node (13) at (1,0) [circle, draw, fill=white] {13};
\node (14) at (0,-1) [circle, draw, fill=white] {14};
\node (124) at (-2,-2) [circle, draw, fill=white] {124};
\node (134) at (2,-2) [circle, draw, fill=white] {134};

\draw[-] (2) to (3);
\draw[-] (2) to (23);
\draw[-] (2) to (24);
\draw[-] (2) to (124);
\draw[-] (3) to (23);
\draw[-] (3) to (13);
\draw[-] (3) to (134);
\draw[-] (124) to (24);
\draw[-] (124) to (14);
\draw[-] (134) to (13);
\draw[-] (134) to (14);
\draw[-] (124) to (134);
\draw[-] (24) to (23);
\draw[-] (23) to (13);
\draw[-] (13) to (14);
\draw[-] (14) to (24);

\begin{pgfonlayer}{background}
\fill[gray!30] (-2,2) to (2,2) to (0,1);
\fill[gray!30] (-2,2) to (0,1) to (-1,0);
\fill[gray!30] (-2,2) to (-1,0) to (-2,-2);
\fill[gray!30] (2,2) to (0,1) to (1,0);
\fill[gray!30] (2,2) to (1,0) to (2,-2);
\fill[gray!30] (-2,-2) to (-1,0) to (0,-1);
\fill[gray!30] (2,-2) to (1,0) to (0,-1);
\fill[gray!30] (-2,-2) to (2,-2) to (0,-1);
\end{pgfonlayer}

\node (2) at (4, 2) [circle, draw, fill=white] {2};
\node (3) at (8,2) [circle, draw, fill=white] {3};
\node (23) at (6,1) [circle, draw, fill=white] {23};
\node (24) at (5,0) [circle, draw, fill=white] {24};
\node (13) at (7,0) [circle, draw, fill=white] {13};
\node (14) at (6,-1) [circle, draw, fill=white] {14};
\node (124) at (4,-2) [circle, draw, fill=white] {124};
\node (134) at (8,-2) [circle, draw, fill=white] {134};

\draw[-] (2) to (3);
\draw[-] (2) to (23);
\draw[-] (2) to (24);
\draw[-] (2) to (124);
\draw[-] (3) to (23);
\draw[-] (3) to (13);
\draw[-] (3) to (134);
\draw[-] (124) to (24);
\draw[-] (124) to (14);
\draw[-] (134) to (13);
\draw[-] (134) to (14);
\draw[-] (124) to (134);
\draw[-] (24) to (23);
\draw[-] (23) to (13);
\draw[-] (13) to (14);
\draw[-] (14) to (24);
\draw[-] (23) to (14);

\begin{pgfonlayer}{background}
\fill[gray!30] (4,2) to (8,2) to (6,1);
\fill[gray!30] (4,2) to (6,1) to (5,0);
\fill[gray!30] (4,2) to (5,0) to (4,-2);
\fill[gray!30] (8,2) to (6,1) to (7,0);
\fill[gray!30] (8,2) to (7,0) to (8,-2);
\fill[gray!30] (4,-2) to (5,0) to (6,-1);
\fill[gray!30] (8,-2) to (7,0) to (6,-1);
\fill[gray!30] (4,-2) to (8,-2) to (6,-1);
\fill[gray!30] (5,0) to (6,1) to (6,-1);
\fill[gray!30] (6,-1) to (6,1) to (7,0);
\end{pgfonlayer}

\end{tikzpicture}

\caption{The simplicial complexes $\hat{\Delta}_{ss}(4)$ and $\hat{\Delta}_{ws}(4)$, respectively.}
\end{figure}
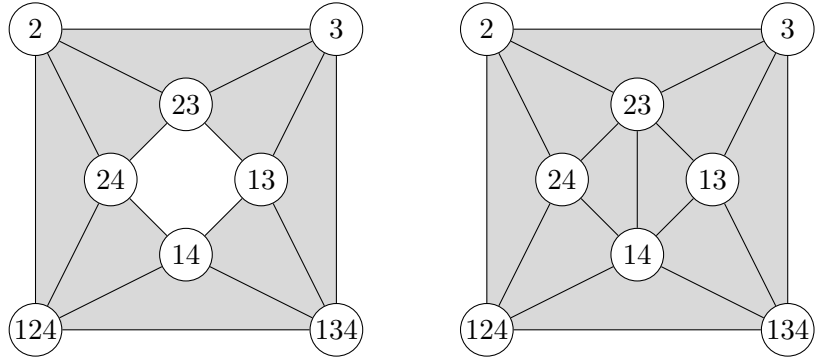

(Note that in these figures, as well as throughout this paper, we omit the braces and commas when referring to subsets of $[n]$. For example, instead of $\{1,2,3,4\}\subset[5]$ we write $1234\subset[5]$.)

Letting $\alpha$ denote the action of set complementation on subsets of $[n]$ and $w_0$ denote the action on subsets of $[n]$ which replaces each $k\in[n]$ by $n+1-k$, we note that both actions respect the relations of strong and weak separation, so that they induce symmetries on $\hat{\Delta}_{ss}(n)$ and $\hat{\Delta}_{ws}(n)$. We let $G=\left<\alpha,w_0\right>\cong \mathbf{Z}/2\mathbf{Z}\times\mathbf{Z}/2\mathbf{Z}$ be the group generated by these symmetries. 

In \textsection3 and \textsection4 of this paper, we prove the following two theorems, answering questions originally posed by V. Reiner:

\begin{thm}\label{ws}
The simplicial complex $\hat{\Delta}_{ws}(n)$ is $G$-contractible for $n\geq4$.
\end{thm}

\begin{thm}\label{ss}
The simplicial complex $\hat{\Delta}_{ss}(n)$ is $G$-homotopy equivalent to the $(n-3)$-sphere $S^{n-3}$.
\end{thm}

The action of $G$ on $\hat{\Delta}_{ss}(n)$ corresponds to an action on the sphere $S^{n-3}$ (with the usual embedding in $\mathbf{R}^{n-2}$) where $\alpha$ acts as the antipodal map and where $w_0$ acts by permuting the axes as follows: if we use the labels $x_2$ through $x_{n-1}$, we replace each $x_k$ with $x_{n+1-k}$.

To prove Theorem \ref{ws} we will formulate and apply an equivariant nerve lemma to a suitable covering of $\hat{\Delta}_{ws}(n)$. To prove Theorem \ref{ss} we will find a $G$-equivariant deformation retraction onto a subcomplex of $\hat{\Delta}_{ss}(n)$ that is the boundary of an $(n-2)$-dimensional cross-polytope, giving a $G$-equivariant homotopy equivalence between $\hat{\Delta}_{ss}(n)$ and $S^{n-3}$. In \textsection5, we consider the question of determining the homeomorphism types of $\hat{\Delta}_{ss}(n)$ and $\hat{\Delta}_{ws}(n)$. 

\section{Preliminaries}

\subsection{Simplicial Complexes}

A \emph{simplicial complex} $\Delta$ is a nonempty collection of finite sets $\sigma$ called \emph{faces} such that if $\sigma\in\Delta$ and $\tau\subset\sigma$, then $\tau\in\Delta$. The singleton subsets $v\in\Delta$ are called the \emph{vertices} of $\Delta$, and the maximal (with respect to inclusion) faces are called the \emph{facets} of $\Delta$. We say that $\Delta$ is \emph{pure} if all facets of $\Delta$ have the same dimension. We denote by $|\Delta|$ the \emph{geometric realization} of $\Delta$ and by $|\sigma|$ the geometric realization of a face $\sigma\in\Delta$.

We will use several constructions involving simplicial complexes: 

\begin{itemize}
\item Given a graph $G$, the \emph{clique complex} of $G$ is the simplicial complex whose faces are precisely the cliques (complete subgraphs) in $G$. 
\item Given a simplicial complex $\Delta$, the \emph{face poset} of $\Delta$, denoted $F(\Delta)$, is the set of faces $\sigma$ of $\Delta$ ordered by inclusion. 
\item Conversely, given a poset $P$, the \emph{order complex} of $P$, denoted $\Delta(P)$, is the simplicial complex whose vertices are the elements of $P$ and whose faces are the chains in $P$. For any simplicial complex $\Delta$, the order complex of the face poset of $\Delta$ is called the \emph{barycentric subdivision} of $\Delta$ and is denoted by $\Sd(\Delta)$. Hence the vertices of $\Sd(\Delta)$ are the faces of $\Delta$ and the faces of $\Sd(\Delta)$ are chains $\tau=\{\sigma_1\subset\cdots\subset\sigma_m\}$ of faces $\sigma_i$ of $\Delta$. We note that there is a natural homeomorphism $|\Delta|\cong|\Sd(\Delta)|$.
\end{itemize}

Let $\Delta$ be a simplicial complex. For a face $\sigma\in\Delta$ we define subcomplexes of $\Delta$ called the the \emph{star}, \emph{deletion}, and \emph{link} of $\sigma$ by 
\begin{align*}
\st(\sigma)&=\{\tau\in\Delta : \tau\cup\sigma\in\Delta\},\\
\dl(\sigma)&=\{\tau\in\Delta : \tau\cap\sigma=\varnothing\},\\
\lk(\sigma)&=\{\tau\in\Delta : \tau\cup\sigma\in\Delta \text{ and } \tau\cap\sigma=\varnothing\},
\end{align*}
respectively. The following proposition is a useful relationship between stars in a clique complex:

\begin{prop}
Let $\Delta$ be the clique complex of a graph $G$ and let $\sigma,\tau\in\Delta$ be faces such that $\sigma\cup\tau\in\Delta$. Then $\st(\sigma)\cap\st(\tau)=\st(\sigma\cup\tau)$. 
\end{prop}

\begin{proof} 
If $\rho$ is a set whose elements form a clique in $G$, then $\rho\cup\sigma\cup\tau$ forms a clique if and only if $\rho\cup\sigma$, $\rho\cup\tau$, and $\sigma\cup\tau$ form cliques. By assumption $\sigma\cup\tau$ forms a clique, and hence $\rho\cup\sigma\cup\tau$ is a face of $\Delta$ if and only if $\rho\cup\sigma$ and $\rho\cup\tau$ are faces. Therefore $\rho\in\st(\sigma\cup\tau)$ if and only if $\rho\in\st(\sigma)\cap\st(\tau)$, as desired. 
\end{proof}

A \emph{cone point} of a simplicial complex $\Delta$ is defined as a vertex $v\in\Delta$ such that $\sigma\cup\{v\}$ is a face of $\Delta$ for any face $\sigma\in\Delta$. We say that a face $\tau\in\Delta$ is a \emph{cone face} if all elements of $\tau$ are cone points of $\Delta$; that is, if $\sigma\cup\tau$ is a face of $\Delta$ for all faces $\sigma\in\Delta$. In particular, note that any simplicial complex with a cone point (or cone face) is contractible by a straight-line homotopy to the cone point (or to any point on the geometric realization of the cone face). Hence any star $\st(\sigma)$ is contractible, since any vertex $v\in\sigma$ is a cone point of $\st(\sigma)$.

Finally, we recall the usual statement of the nerve lemma:

\begin{lem}[Nerve Lemma] \cite[10.6]{Bjorner}
Let $\Delta$ be a simplicial complex and let $\{\Delta_i\}_{i\in I}$ be a family of subcomplexes such that $\Delta=\bigcup_{i\in I}\Delta_i$. If every nonempty finite intersection $\Delta_{i_1}\cap\Delta_{i_2}\cap\cdots\cap\Delta_{i_k}$ is contractible, then $\Delta$ and the nerve $\mathcal{N}(\{\Delta_i\})$ are homotopy equivalent.
\end{lem}

Here, the \emph{nerve} $\mathcal{N}(\{\Delta_i\})$ of the covering $\{\Delta_i\}_{i\in I}$ is defined as the simplicial complex on the vertex set $I$ such that a finite subset $\sigma\subset I$ is a face of $\mathcal{N}(\{\Delta_i\})$ if and only if $\bigcap_{i\in\sigma}\Delta_i\neq\varnothing$.

\subsection{Equivariant Tools}

In the following, we let $G$ be any group. 

A \emph{$G$-simplicial complex} is a simplicial complex together with an action of $G$ on the vertices that takes faces to faces. A \emph{$G$-topological space} (or \emph{$G$-space}) is a topological space together with a continuous action of $G$. A map $f: X \to Y$ between $G$-topological spaces is called \emph{$G$-continuous} (or a \emph{$G$-map}) if it is continuous and respects the action of $G$ so that $f(gx)=gf(x)$ for all $x\in X$. 

A \emph{contractible carrier} from a simplicial complex $\Delta$ to a topological space $X$ is an inclusion-preserving map sending faces of $\Delta$ to contractible subspaces of $X$. We say that a map $f\colon |\Delta|\to X$ is \emph{carried} by a contractible carrier $C$ if $f(|\sigma|)\subset C(\sigma)$ for each face $\sigma\in\Delta$. For $\sigma$ a face of a $G$-simplicial complex $\Delta$, we denote by $G_{\sigma}$ the subgroup $\{g\in G: g\sigma=\sigma\}$.

In \cite{Webb}, Th\'evenaz and Webb prove equivariant formulations of the contractible carrier lemma and the Quillen fiber lemma:

\begin{lem}\cite[Lemma 1.5(b)]{Webb}\label{webbcarry}
Let $\Delta$ be a $G$-simplicial complex such that for each face $\sigma\in\Delta$, $G_{\sigma}$ fixes the vertices of $\sigma$. Let $X$ be a $G$-space, and let $C$ be a contractible carrier from $\Delta$ to $X$ such that $C(g\sigma)=gC(\sigma)$ for all $g\in G$ and $\sigma\in\Delta$, and such that $G_{\sigma}$ acts trivially on $C(\sigma)$ for all $\sigma\in\Delta$. Then any two $G$-maps $|\Delta|\to X$ that are both carried by $C$ are $G$-homotopic.
\end{lem}

\begin{lem}\cite[Theorem 1]{Webb}\label{webbquillen}
Let $P$ and $Q$ be $G$-posets, and let $f\colon P\to Q$ be a mapping of $G$-posets. Suppose that for all $q\in Q$ the fiber $f^{-1}(Q_{\geq q})$ is $G_q$-contractible, or for all $q\in Q$ the fiber $f^{-1}(Q_{\leq q})$ is $G_q$-contractible. Then $f$ induces a $G$-homotopy equivalence between the order complexes $\Delta(P)$ and $\Delta(Q)$.
\end{lem}

In this section, we slightly generalize the former in Lemma \ref{gcarry}, and use the latter to prove an equivariant formulation of the nerve lemma in Lemma \ref{gnerve}. We note that the result of Lemma \ref{gcarry} also appears in \cite[Satz 1.5]{Welker}, and that a stronger formulation of Lemma \ref{gnerve} appears in \cite{Yang}. The formulation of the equivariant nerve lemma that we give in Lemma \ref{gnerve} serves our purposes. Finally, we prove in Proposition \ref{conecontract} that any $G$-simplicial complex with a cone point is $G$-contractible.

Our slightly more general equivariant contractible carrier lemma is the following:

\begin{lem}\label{gcarry}
If $C$ is a contractible carrier from a $G$-simplicial complex $\Delta$ to a $G$-space $X$ such that $C(g\sigma)=gC(\sigma)$ for each $g\in G$ and each face $\sigma\in\Delta$, and such that each $C(\sigma)$ is $G_{\sigma}$-contractible, then any two $G$-maps carried by $C$ are $G$-homotopic. 
\end{lem}

\begin{proof}
We define a carrier from the barycentric subdivision $\Sd(\Delta)$ of $\Delta$ to $X$ as follows: If $\tau=\{\sigma_1\subset\cdots\subset\sigma_m\}$ is a face of $\Sd(\Delta)$, where each $\sigma_i$ is a face of $\Delta$, let $C'(\tau)$ be the subset of points in $C(\sigma_m)$ that are fixed by each element of $G_{\tau}$. This contracts to a point by the restriction of the homotopy that $G_{\sigma_m}$-contracts $C(\sigma_m)$ to a point. Also, if $\tau\subset \tau'$ then $\sigma_m\subset\sigma'_m$, so that $C(\sigma_m)\subset C(\sigma'_m)$ and $G_{\tau'}\subset G_{\tau}$. Hence $C'(\tau)$, which is the subset of points in $C(\sigma_m)$ fixed by $G_{\tau}$, is contained in the subset of points in $C(\sigma'_m)$ fixed by $G_{\tau}$, which is contained in the subset of points in $C(\sigma'_m)$ fixed by $G_{\tau'}$. This is $C'(\tau')$, so that $C'$ is a contractible carrier. 

We have for any face $\tau\in \Sd(\Delta)$ that $G_\tau$ fixes the vertices of $\tau$ (because it sends each face of $\Delta$ to a face of the same dimension) and that $C'(g\tau)=gC'(\tau)$ and that $G_{\tau}$ acts trivially on $C'(\tau)$, so we may apply Lemma \ref{webbcarry}: Any two $G$-maps $|\Sd(\Delta)|\to X$ that are carried by $C'$ are $G$-homotopic. But any $G$-map $|\Delta|\cong |\Sd(\Delta)|\to X$ that is carried by $C$ is also carried by $C'$: A face of $\Sd(\Delta)$, call it again $\tau=\{\sigma_1\subset\cdots\subset \sigma_m\}$, is fixed pointwise by $G_{\tau}$, and $|\tau|\subset|\sigma_m|$, so that the image of $|\tau|$ under a $G$-map carried by $C$ is contained in $C(\sigma_m)$, so that the image of $|\tau|$ under a $G$-map carried by $C$ is contained in $C'(\tau)$, so that any $G$-map carried by $C$ is also carried by $C'$, as desired.
\end{proof}

To formulate an equivariant nerve lemma, we need a condition on a cover of a $G$-simplicial complex which will allow the action of $G$ on the entire simplicial complex to induce an action on the nerve of the covering.

\begin{defn}
For a $G$-simplicial complex $\Delta$, we say that a covering $\{\Delta_i\}_{i\in I}$ of $\Delta$ by subcomplexes is \emph{$G$-invariant} if for all $i\in I$ and for all $g\in G$ there exists a unique $j\in I$ such that $g\Delta_i=\Delta_j$. 
\end{defn}

Given such a cover of a $G$-simplicial complex, we see that the action of $G$ on $\Delta$ induces an action on the index set $I$, defined by letting $gi=j$ when $g\Delta_i=\Delta_j$. We claim that this action of $G$ sends faces to faces in the nerve $\mathcal{N}(\{\Delta_i\})$. If $\sigma$ is a face of $\mathcal{N}(\{\Delta_i\})$, then the intersection $\bigcap_{i\in\sigma}\Delta_i$ is nonempty. Hence 
$$g\bigcap_{i\in\sigma}\Delta_i=\bigcap_{i\in\sigma}g\Delta_i=\bigcap_{j\in g\sigma}\Delta_j$$
is also nonempty, so that $g\sigma$ is a face of $\mathcal{N}(\{\Delta_i\})$. Thus the action of $G$ on $\Delta$ induces a simplicial action on the nerve $\mathcal{N}(\{\Delta_i\})$. This induced action allows us to formulate an equivariant nerve lemma as follows:

\begin{lem}\label{gnerve}
Let $\Delta$ be a $G$-simplicial complex and let $\{\Delta_i\}_{i\in I}$ be a $G$-invariant covering of $\Delta$. If every nonempty finite intersection $\bigcap_{i\in\sigma}\Delta_i$, where $\sigma\subset I$, is $G_{\sigma}$-contractible, then $\Delta$ and the nerve $\mathcal{N}(\{\Delta_i\})$ are $G$-homotopy equivalent.
\end{lem}

\begin{proof} 
Let $Q=F(\Delta)$ and $P=F(\mathcal{N}(\{\Delta_i\}))$ be the face posets of $\Delta$ and the nerve $\mathcal{N}(\{\Delta_i\})$, respectively. Define a map $f\colon Q\to P$ by $\pi\mapsto\{i \in I : \pi \in \Delta_i\}$. This is order preserving and it is also $G$-equivariant: For any $\pi\in Q$ and $g\in G$ we have
\begin{align*}
gf(\pi)&= g\{i \in I : \pi \in \Delta_i\}\\
&=\{j \in I : g\Delta_i = \Delta_j \text{ for some $\Delta_i$ containing $\pi$}\}\\
&=\{j \in I : g\pi \in \Delta_j\}\\
&=f(g\pi).
\end{align*}
In addition, for any $\sigma\in P$, the fibers $f^{-1}(P_{\geq\sigma}) = \bigcap_{i\in \sigma}\Delta_i$ are $G_{\sigma}$-contractible by hypothesis. The result now follows from an application of Lemma \ref{webbquillen}.
\end{proof}

Finally, we have the following useful condition for $G$-contractibility of a $G$-simplicial complex:

\begin{prop}\label{conecontract}
A $G$-simplicial complex $\Delta$ with a cone point $v\in\Delta$ is $G$-contractible.
\end{prop}

\begin{proof}
As $v$ is a cone point of $\Delta$, we must have that the $G$-orbit of $v$, $Gv$, is a cone face of $\Delta$. The straight-line homotopy to the barycenter of $Gv$ respects the action of $G$, so that $\Delta$ is $G$-contractible. 
\end{proof}

\section{Proof of Theorem \ref{ws}}

In this section, we assume that $n\geq4$.

Recall the action of the group $G=\left<\alpha, w_0\right>$ on $\hat{\Delta}_{ws}(n)$ as defined in \textsection1. To prove Theorem \ref{ws}, we claim that the family $\{\Delta_i\}_{i\in I}$ composed of the following subcomplexes of $\hat{\Delta}_{ws}(n)$:
$$\dl(2),\; \dl([n]\smallsetminus 2),\; \dl(3),\; \dl([n]\smallsetminus 3),\; \dots,\; \dl(n-1),\; \dl([n]\smallsetminus(n-1))$$
is a $G$-invariant covering to which we may apply the equivariant nerve lemma (Lemma \ref{gnerve}). We show this with two lemmas:

\begin{lem}
The family of subcomplexes $\{\Delta_i\}_{i\in I}$ is a $G$-invariant covering of $\hat{\Delta}_{ws}(n)$ with nerve $\mathcal{N}(\{\Delta_i\})$ a simplex.
\end{lem}

\begin{proof}
Any face of $\hat{\Delta}_{ws}(n)$ either does not contain some singleton $k\in\{2,3,\dots,n-1\}$, in which case it lies in $\dl(k)$, or it does, in which case it cannot also contain its complement $[n]\smallsetminus k$ and therefore lies in $\dl([n]\smallsetminus k)$. Hence these subcomplexes cover $\hat{\Delta}_{ws}(n)$. It is also clear that this covering is $G$-invariant. 

To see that $\mathcal{N}(\{\Delta_i\})$ is a simplex, notice that for any subset $\sigma\subset I$ the intersection $\bigcap_{i\in\sigma}\Delta_i$ contains the subcomplex $\st(\{1n, 23\cdots n-1\})\subset\hat{\Delta}_{ws}(n)$ (as this subcomplex consists of exactly those faces of $\hat{\Delta}_{ws}(n)$ which contain no singletons and no complements of singletons) and is thus nonempty. 
\end{proof}

Given Lemma 3.1, in order to apply Lemma 2.5 to the covering $\{\Delta_i\}_{i\in I}$ to conclude that $\hat{\Delta}_{ws}(n)$ is $G$-contractible, it remains to verify the following:

\begin{lem}
For every subset $\sigma\subset I$, the intersection $\bigcap_{i\in\sigma}\Delta_i$ is $G_{\sigma}$-contractible.
\end{lem}

To prove this, we induct on the number of \emph{free complementary pairs} of an intersection $\bigcap_{i\in\sigma}\Delta_i$, meaning the number of pairs of subsets $(k, [n]\smallsetminus k)$ such that neither $\dl(k)$ nor $\dl([n]\smallsetminus k)$ is involved in $\bigcap_{i\in\sigma}\Delta_i$.

\begin{proof}
For the base case, suppose that we are given an intersection 
$$\bigcap_{i\in\sigma}\Delta_i=\big(\dl(k_1)\cap\dl(k_2)\cap\cdots\cap\dl(k_m)\big)\cap\big(\dl([n]\smallsetminus\ell_1)\cap\dl([n]\smallsetminus\ell_2)\cap\cdots\cap\dl([n]\smallsetminus\ell_n)\big)$$ 
which has no free complementary pairs. We claim that the family of subcomplexes $\{K_j\}_{j\in J}$ which is composed of the subcomplexes
$$\st(1n),\; \st(23\cdots n-1),$$
$$\{\st(r)\}_{\substack{r\in\{2,3,\dots, n-1\}\smallsetminus\{k_1, k_2, \dots, k_m\}}},$$
$$\{\st([n]\smallsetminus s)\}_{\substack{s\in\{2,3,\dots, n-1\}\smallsetminus\{\ell_1, \ell_2,\dots, \ell_n\}}}$$
is a $G_{\sigma}$-invariant cover of $\bigcap_{i\in\sigma}\Delta_i$ to which we may apply Lemma \ref{gnerve}. (Note: here, we are taking the stars within the subcomplex $\bigcap_{i\in\sigma}\Delta_i$. We may do this because each of the vertices whose stars are in this cover do indeed lie in $\bigcap_{i\in\sigma}\Delta_i$ by hypothesis.)

To see that the subcomplexes $K_j$ cover $\bigcap_{i\in\sigma}\Delta_i$, note that any face of $\bigcap_{i\in\sigma}\Delta_i$ which is not in $\st(1n)$ must contain some singleton $r\not\in\{k_1, k_2, \dots, k_m\}$, and therefore lies in $\st(r)$, which will be in the cover. Moreover, the cover $\{K_j\}_{j\in J}$ is $G_{\sigma}$-invariant as each element of the subgroup $G_{\sigma}$ of $G$ either fixes or interchanges $1n$ and $23\cdots n-1$, and thus either preserves or interchanges their stars. Each element of $G_{\sigma}$ sends singletons and complements of singletons that are not contained in $\bigcap_{i\in\sigma} \Delta_i$ to one another, and thus sends elements of 
$$\big(\{2,\dots,n-1\}\smallsetminus\{k_1,\dots,k_m\}\big)\cup\big(\{[n]\smallsetminus2,\dots,[n]\smallsetminus(n-1)\}\smallsetminus\{[n]\smallsetminus \ell_1,\dots,[n]\smallsetminus\ell_n\}\big)$$
to one another. Hence each element of $G_{\sigma}$ sends their stars to one another as well. 

Finally, we show that for every subset $\tau\subset J$, the intersection $\bigcap_{j\in\tau}K_j$ is $(G_{\sigma})_{\tau}$-contractible. We separate this into a few simple cases (that may overlap). In each case, we exhibit a cone point for the intersection which, together with the fact that $\bigcap_{j\in\tau}K_j$ will always be a $(G_{\sigma})_{\tau}$-simplicial complex, allows us to apply Proposition \ref{conecontract} to show $(G_{\sigma})_{\tau}$-contractibility.

\begin{itemize}

\item\textbf{Case 1:} \emph{For no $j\in\tau$ is $K_j=\st(1n)$ or $K_j=\st(23\cdots n-1)$.}

\noindent Because $\bigcap_{i\in\sigma}\Delta_i$ has no free complementary pairs, for any singleton $r$ whose star is involved in $\bigcap_{j\in\tau}K_j$ we know that no $K_j=\st([n]\smallsetminus r)$. Therefore $\bigcap_{j\in\tau}K_j$ is an intersection of stars of subsets which are all pairwise weakly separated, and hence $\bigcap_{j\in\tau}K_j$ is the star of a face of $\bigcap_{i\in\sigma}\Delta_i$, which has a cone point.

\smallskip

\item\textbf{Case 2:} \emph{For some $j\in\tau$ we have $K_j=\st(1n)$.}

\smallskip

\begin{itemize}

\item\textbf{Subcase 2.1:} \emph{For at least two $j\in\tau$ we have that $K_j$ is the star of a singleton.}

Order the singletons whose stars are involved in $\bigcap_{j\in\tau}K_j$ so that $r_1$ and $r_2$ are the two least such singletons. We claim that the vertex $r_1r_2$ is a cone point of $\bigcap_{j\in\tau} K_j$. Because the subset $r_1r_2$ is weakly separated from $1n$ and $23\cdots n-1$, it lies in both of their stars. It is also weakly separated from every singleton whose star is involved in the intersection since $r_1$ and $r_2$ are the least two such singletons. As there are no free complementary pairs in the intersection $\bigcap_{i\in\sigma}\Delta_i$, the stars of neither $[n]\smallsetminus r_1$ nor $[n]\smallsetminus r_2$ can be involved in the intersection $\bigcap_{j\in\tau}K_j$, so that $r_1r_2$ is also weakly separated from each complement of a singleton whose star is involved in the intersection. Thus the vertex $r_1r_2$ lies in each $K_j$, and therefore lies in $\bigcap_{j\in\tau}K_j$. Finally, because any face of $\bigcap_{j\in\tau}K_j$ consists of subsets which are weakly separated from both $r_1$ and $r_2$ by hypothesis and also contains no singletons (because no singleton is weakly separated from $1n$), each subset in a face of $\bigcap_{j\in\tau}K_j$ is weakly separated from the subset $r_1r_2$. Therefore $r_1r_2$ is a cone point, as desired.

\smallskip

\item\textbf{Subcase 2.2:} \emph{For exactly one $j\in\tau$ we have that $K_j$ is the star of a singleton.}

Let $r$ be the singleton in question. At least one of the subsets $1r$ or $rn$ is not frozen, without loss of generality let it be $1r$. The subset $1r$ is weakly separated from both $1n$ and $23\cdots n-1$, as well as from the singleton $r$ and the complement of any singleton that is not equal to $r$, and thus lies in their stars. Because there are no free complementary pairs in the intersection $\bigcap_{i\in\sigma}\Delta_i$, the star $\st([n]\smallsetminus r)$ cannot also be involved in the intersection $\bigcap_{j\in\tau}K_j$, and so the vertex $1r$ lies in the intersection $\bigcap_{j\in\tau} K_j$. Finally, because any face of $\bigcap_{j\in\tau}K_j$ consists of subsets which are weakly separated from $r$ and also contains no singletons, each subset in a face of $\bigcap_{j\in\tau}K_j$ is weakly separated from the subset $1r$. Therefore $1r$ is a cone point, as desired.

\smallskip

\item\textbf{Subcase 2.3:} \emph{For no $j\in\tau$ is $K_j$ the star of a singleton.}

Because the subset $1n$ is weakly separated from $23\cdots n-1$ as well as every complement of a singleton, $\bigcap_{i\in\sigma}K_j$ is the star of a face of $\bigcap_{i\in\sigma}\Delta_i$, which has a cone point.

\end{itemize}

\smallskip

\item\textbf{Case 3:} \emph{For some $j\in\tau$ we have $K_j=\st(23\cdots n-1)$.}

\smallskip

\begin{itemize}

\item\textbf{Subcase 3.1:} \emph{For at least two $j\in\tau$ we have that $K_j$ is the star of a complement of a singleton.}

An argument symmetric to subcase 2.1 above shows that $\bigcap_{j\in\tau}K_j$ has a cone point $[n]\smallsetminus s_1s_2$, where $s_1$ and $s_2$ are the two least singletons such that $[n]\smallsetminus s_1$ and $[n]\smallsetminus s_2$ are involved in $\bigcap_{j\in\tau}K_j$.

\smallskip

\item\textbf{Subcase 3.2:} \emph{For exactly one $j\in\tau$ we have that $K_j$ is the star of a complement of a singleton.}

An argument symmetric to subcase 2.2 shows that $\bigcap_{j\in\tau}K_j$ has a cone point $[n]\smallsetminus 1s$ or $[n]\smallsetminus sn$, where $s$ is the only singleton whose complement is involved in $\bigcap_{j\in\tau}K_j$.

\smallskip

\item\textbf{Subcase 3.3:} \emph{For no $j\in\tau$ is $K_j$ the star of a complement of a singleton.}

Because the subset $23\cdots n-1$ is weakly separated from $1n$ as well as every singleton, $\bigcap_{i\in\sigma}K_j$ is the star of a face of $\bigcap_{i\in\sigma}\Delta_i$, which has a cone point.

\end{itemize}

\end{itemize}

We conclude that we may apply Lemma \ref{gnerve} to show that $\bigcap_{i\in\sigma}\Delta_i$ is $G_{\sigma}$-homotopy equivalent to the nerve $\mathcal{N}(\{K_j\})$, which is a simplex, and therefore contractible.

For the inductive step, let $(k, [n]\smallsetminus k)$ be a free complementary pair of $\bigcap_{i\in\sigma}\Delta_i$. If we have $G_{\sigma}=\{e\}$ or $\left<\alpha\right>$, then
$$\Big(\bigcap_{i\in\sigma}\Delta_i\Big)\cap\dl(k),\; \Big(\bigcap_{i\in\sigma}\Delta_i\Big)\cap\dl([n]\smallsetminus k)$$
is a $G_{\sigma}$-invariant cover of $\bigcap_{i\in\sigma}\Delta_i$ for the same reason that the $\Delta_i$ cover $\hat{\Delta}_{ws}(n)$; if we have $G_{\sigma}=\left<w_0\right>$, $\left<\alpha w_0\right>$, or $G$, then
$$\Big(\bigcap_{i\in\sigma}\Delta_i\Big)\cap\dl(k),\; \Big(\bigcap_{i\in\sigma}\Delta_i\Big)\cap\dl([n]\smallsetminus k),$$
$$\Big(\bigcap_{i\in\sigma}\Delta_i\Big)\cap\dl(n+1-k),\; \Big(\bigcap_{i\in\sigma}\Delta_i\Big)\cap\dl([n]\smallsetminus(n+1-k))$$
is a $G_{\sigma}$-invariant cover of $\bigcap_{i\in\sigma}\Delta_i$. In either case, every intersection of the subcomplexes in the cover has at least one fewer free complementary pair. Therefore, by induction, each such intersection is nonempty and equivariantly contractible.
\end{proof}

\section{Proof of Theorem \ref{ss}}

In this section, we again assume that $n\geq 4$.

Recall that the group $G=\left<\alpha,w_0\right>$ is also a group of symmetries of the simplicial complex $\hat{\Delta}_{ss}(n)$. We begin to prove Theorem \ref{ss} by defining the following subcomplex of $\hat{\Delta}_{ss}(n)$:

\begin{defn}
We let $K\subset\hat{\Delta}_{ss}(n)$ be the vertex-induced subcomplex whose vertices are the singleton subsets of $[n]$ and their complements.
\end{defn}

For example, see the left side of Figure 2, in which $K$ is the square formed by the vertices $2$, $3$, $124$, and $134$.

\begin{prop}
The subcomplex $K$ is simplicially isomorphic to the boundary of an $(n-2)$-dimensional cross polytope.
\end{prop}

\begin{proof}
The boundary of an $(n-2)$-dimensional cross polytope is the clique complex of a graph of $n-2$ pairs of antipodal vertices, each of which is connected to every vertex except for its antipode. The subcomplex $K$ has $n-2$ pairs of complementary vertices (the singleton $k$ and its complement $[n]\smallsetminus k$), each of which is separated from every vertex except for its complement.
\end{proof}

As a cross polytope is homeomorphic to a ball, its boundary is homeomorphic to a sphere, so that $K$ is homeomorphic to $S^{n-3}$. We note that $G$ preserves $K$, and acts on it as follows: $\alpha$ acts by set complementation, which passes to the antipodal map on $S^{n-3}$, and $w_0$ acts by exchanging each singleton $k$ with $n+1-k$ and each complement $[n]\smallsetminus k$ with $[n]\smallsetminus(n+1-k)$, which, if $S^{n-3}$ is given the usual embedding in $\mathbf{R}^{n-2}$ with the axes labelled as $x_2$ through $x_{n+1-k}$, passes to permuting the axes by exchanging each $x_k$ with $x_{n+1-k}$. 

To prove Theorem \ref{ss}, we will define a $G$-map $\pi\colon|\hat{\Delta}_{ss}(n)|\to|K|$ that we will prove to be a $G$-deformation retraction by a contractible carrier argument. 

\subsection{Defining a map $\pi\colon|\hat{\Delta}_{ss}(n)|\to|K|$}

The following lemma will allow us to indirectly define a map $\pi$:

\begin{lem}\label{inducecontsmap}
Assume we have a function $\pi'\colon \Sd(\hat{\Delta}_{ss}(n))\to \Sd(K)$ defined on the vertices of $\Sd(\hat{\Delta}_{ss}(n))$ and with the following property: For each face $\tau=\{\sigma_1\subset\cdots\subset\sigma_m\}$ of $\Sd(\hat{\Delta}_{ss}(n))$, where each $\sigma_i$ is a face of $\hat{\Delta}_{ss}(n)$, we have that $\bigcup_{i=1}^m\pi'(\sigma_i)$ is a face of $K$. Then $\pi'$ induces a map $\pi\colon|\hat{\Delta}_{ss}(n)| \to|K|$. If $\pi'$ is $G$-equivariant, then so is $\pi$.
\end{lem}

\begin{proof}
We recall that there are natural homeomorphisms between the geometric realizations of a simplicial complex and its barycentric subdivision, $|\hat{\Delta}_{ss}(n)|\cong |\Sd(\hat{\Delta}_{ss}(n))|$ and $|K|\cong |\Sd(K)|$. Thus we may define a map $|\hat{\Delta}_{ss}(n)|\to |K|$ by defining first where the vertices of $\Sd(\hat{\Delta}_{ss}(n))$ are sent, and then sending any convex combination of some vertices that form a face to the corresponding convex combination of their images in $|K|$. We need only check that those convex combinations exist; that is, for any face $\tau=\{\sigma_1\subset\cdots\subset\sigma_m\}$ that the images of each vertex $\sigma_i\in\tau$ lie on the geometric realization of the same face $\nu\in K$.

We have that the vertex of the barycentric subdivision $\nu'\in \Sd(K)$ lies in the geometric realization of a face of $K$, $|\nu|\subset |K|$, if and only if the vertex set of $\nu'$ is contained in the vertex set of $\nu$. Thus if $\bigcup_{\sigma_i\in\tau} \pi'(\sigma_i)$ is a face of $K$, then all of the $\pi'(\sigma_i)$ will lie on the same face, so that we will have defined a map $\pi\colon|\hat{\Delta}_{ss}(n)|\to |K|$, as desired.

Finally, we remark that if $\pi'$ respects the action of $G$, then so does the induced map $\pi$, because taking corresponding convex combinations respects the action of $G$. 
\end{proof}

We now define a map $\pi'\colon \Sd(\hat{\Delta}_{ss}(n))\to \Sd(K)$ on the vertices of the barycentric subdivision to which we will be able to apply the preceding lemma. For $\sigma$ a face of $\hat{\Delta}_{ss}(n)$ and for $v$ a vertex of $K$, we let $v\in \pi'(\sigma)$ if and only if $\sigma\cup\{v\}$ is a face of $\hat{\Delta}_{ss}(n)$ but $\sigma\cup\{\alpha(v)\}$ is not. By design, $\pi'(\sigma)$ will not contain any complementary pairs, so that as long as it is nonempty it will in fact be a face of $K$, i.e. a vertex of $\Sd(K)$. We also note that $\pi'$ respects the action of $G$.

We now prove for any face $\sigma\in \hat{\Delta}_{ss}(n)$ that $\pi'(\sigma)$ is in fact nonempty via the following lemma:

\begin{lem}\label{imagenonempty}
There is no face $\sigma\in\hat{\Delta}_{ss}(n)$ such that for each $k\in \{2, 3, \dots, n-1\}$, we have that either both or neither of $\sigma\cup\{k\}$, $\sigma\cup\{[n]\smallsetminus k\}$ is a face of $\hat{\Delta}_{ss}(n)$. 
\end{lem}

\begin{proof}
We argue by contradiction, assuming $n$ is minimal such that a counterexample $\sigma$ exists. 

Encode a vertex $v$ of $\hat{\Delta}_{ss}(n)$ --- which may also be viewed as a subset of $[n]$ --- as a sequence of $n$ $0$s and $1$s, with a $0$ in the $k$th slot if $k\notin v$ and a $1$ in the $k$th slot if $k\in v$. For example, $001100111$ corresponds to $34789\subset[9]$. Each segment of zeros or ones in the sequence is either initial, final, or interior. We say that a slot is initial, final, or interior if it lies in an initial, final, or interior segment of $0$s or $1$s, respectively. Note that two sequences are \emph{not} strongly separated from one another if and only if for some slots $k_1<k_2<k_3$, one sequence restricts to $101$ and the other restricts to $010$.

These sequences clearly indicate the singletons and complements of singletons from which they are strongly separated: 

\begin{itemize}
\item If a slot $j$ is in an initial or final segment of $0$s or $1$s, then both the singleton $\{j\}$ corresponding to the slot and its complement $[n]\smallsetminus\{j\}$ are strongly separated from the vertex corresponding to the sequence. 
\item A singleton $\{j\}$ corresponding to an interior $0$ in slot $j$ will not be strongly separated from the vertex while its complement $[n]\smallsetminus\{j\}$ will be. 
\item A singleton $\{j\}$ corresponding to an interior $1$ in slot $j$ will be strongly separated from the vertex while its complement $[n]\smallsetminus\{j\}$ will not be.
\end{itemize}

In the counterexample $\sigma$, each slot $j$ falls into one of two categories: either both the singleton $\{j\}$ and its complement $[n]\smallsetminus \{j\}$ are separated from every vertex in $\sigma$, i.e. $j$ is always in an initial or final segment, or else some vertex is not strongly separated from the singleton $\{j\}$ and some other vertex is not strongly separated from the complement $[n]\smallsetminus\{j\}$, i.e. one vertex has an interior $1$ in slot $j$ and another has an interior $0$. If there were more than one slot that is always initial or more than one slot that were always final, we could remove one of the extra slots and the collection would remain a counterexample, so that $n$ would not be minimal.

Thus there is only one slot that is always initial and one that is always final, so that for each $j$ in the interval $[2,n-1]$ there is a vertex with a noninitial, nonfinal $1$ in slot $j$. Let $v_j$ be such a vertex, with the segment of $1$s containing slot $j$ extending as far to the right as possible. Let $k_1=2$, and for $k_{i-1}\in[2,n-1]$, let $k_i$ be the slot of the first zero after slot $k_{i-1}$ in $v_{k_{i-1}}$. The sequence of $k_i$ is a strictly increasing sequence of integers, and the fact that the $1$ in slot $k_{i-1}$ of $v_{k_{i-1}}$ is nonfinal means that for $k_{i-1}\in[2,n-1]$, we must have that $k_i\in[2,n]$, with $k_1\in [2,n-1]$, so that we must for some $m$ have $k_1$ through $k_{m-1}$ all in $[2,n-1]$, with $k_m=n$. 

We also must have that $v_{k_i}$ has a zero in some slot in the interval $[k_{i-1},k_i-1]$; otherwise, $v_{k_i}$ would contain the segment $[{k_{i-1}},{k_i}]$ in a noninitial and nonfinal segment of $1$s, which ends further to the right than the ending point of slot ${k_i}-1$ of the corresponding segment in $v_{k_{i-1}}$, contradicting the definition of $v_{k_{i-1}}$. We also need that $v_{k_i}\cap[k_i+1,n]\supset v_{k_{i-1}}\cap[k_i+1,n]$; otherwise, for some slot in $[k_{i-1},k_i-1]$, for the slot $k_i$, and for some slot in $[k_i+1,n]$, we have that $v_{k_{i-1}}$ restricts to $101$ and $v_{k_i}$ restricts to $010$, meaning that they are not strongly separated from one another, a contradiction of the assumption that $v_{k_i}\cap[k_i+1,n]\not\supset v_{k_{i-1}}\cap[k_i+1,n]$.

We have that $k_m=n$, so that $v_{k_{m-1}}$ has a $0$ in slot $n$. By induction, as each $v_{k_i}\cap\{n\}\supset v_{k_{i-1}}\cap \{n\}$ we must have that $v_{k_1}=v_2$ has a $0$ in slot $n$, so that $v_2$ restricts to $101$ in slots $1<2<n$. By a symmetric argument, $\sigma$ must contain a vertex that restricts to $010$ in the same slots, but these two vertices cannot be strongly separated from one another --- a contradiction of the assumption that a counterexample exists.
\end{proof}

Thus we must have that each $\pi'(\sigma)$ is nonempty, so that $\pi'$ is in fact a function from the vertices of $\Sd(\hat{\Delta}_{ss}(n))$ to the vertices of $\Sd(K)$. By Lemma \ref{inducecontsmap}, $\pi'$ will induce a map $\pi\colon|\hat{\Delta}_{ss}(n)|\to|K|$ if for each face $\tau= \{\sigma_1\subset\cdots\subset\sigma_m\}$ of $\Sd(\hat{\Delta}_{ss}(n))$ we have that $\bigcup_i\pi(\sigma_i)$ is a face of $K$.

If a vertex $v$ of $K$ is in $\pi(\sigma_i)$, then we have that $v$ is strongly separated from each vertex of $\sigma_i$, so that for $j<i$, we have that $v$ is strongly separated from each vertex of $\sigma_j\subset \sigma_i$, so that $\alpha(v)$ is not in $\pi(\sigma_j)$.  Also, $\alpha(v)$ is \emph{not} strongly separated from some vertex in $\sigma_i$, which for $j>i$ must also be a vertex of $\sigma_j\supset\sigma_i$, so that $\alpha(v)$ is not in $\pi(\sigma_j)$. Thus $\bigcup_j\pi(\sigma_j)$ does not contain any complementary pairs, and any set of vertices of $K$ not containing complementary pairs forms a face.

\subsection{Proving that $\pi$ is a $G$-deformation retraction}

We let $\iota\colon|K|\hookrightarrow|\hat{\Delta}_{ss}(n)|$ be the inclusion map. Each face of $K$ (i.e. each vertex of $\Sd(K)$) is strongly separated from every vertex of $K$ other than the complements of its vertices, so that $\pi'$ acts as the identity on vertices of $\Sd(K)\subset \Sd(\hat{\Delta}_{ss}(n))$ and thus induces the identity on $|K|\cong|\Sd(K)|\subset |\hat{\Delta}_{ss}(n)|$. Hence $\pi\circ\iota$ is the identity on $|K|$. Thus to show that $\pi$ is a deformation retraction, we need only show that $\iota\circ\pi$ is homotopic to the identity map on $|\hat{\Delta}_{ss}(n)|$.

By Lemma \ref{gcarry}, it is enough to find a valid common contractible carrier for the identity map on $|\Sd(\hat{\Delta}_{ss}(n))|\cong |\hat{\Delta}_{ss}(n)|$ and for $\iota\circ\pi$. If $\tau=\{\sigma_1\subset\cdots\subset \sigma_m\}$ is a face of $\Sd(\hat{\Delta}_{ss}(n))$, let $C(\tau)=\st(\sigma_1)$, taken here as the star of a face of $\hat{\Delta}_{ss}(n)$, rather than the star of a vertex of $\Sd(\hat{\Delta}_{ss}(n))$. For faces $\tau,\tau'\in\Sd(\hat{\Delta}_{ss}(n))$, we have that $\tau\subset\tau'$ implies $\sigma_1\supset\sigma'_1$, which implies $\st(\sigma_1)\subset\st(\sigma'_1)$. Since $\st(\sigma_1)$ is $G_{\sigma_1}$-contractible by a straight line homotopy to the center of $\sigma_1$, with $G_{\tau}\subset G_{\sigma_1}$, we have that $C(\tau)$ is $G_{\tau}$-contractible, with $C(g\tau)=gC(\tau)$.

The identity is carried by $C$ because the image of $|\tau|$ is contained in $|\sigma_m|\subset\st(\sigma_1)$.

For each $v\in \bigcup_i\pi(\sigma_i)$, there is some $\sigma_i\cup\{v\}$ that is a face of $\hat{\Delta}_{ss}(n)$. Thus $\sigma_1\cup\{v\}\subset\sigma_i\cup\{v\}$ is also a face of $\hat{\Delta}_{ss}(n)$.  We know that $\hat{\Delta}_{ss}(n)$ is a clique complex in which $\sigma_1$ and $\bigcup_i\pi(\sigma_i)$ are both cliques, and $\sigma_1\cup\{v\}$ a clique for each $v\in\bigcup_i\pi(\sigma_i)$, so that $\sigma_1\cup\bigcup_i\pi(\sigma_i)$ is a clique, and thus forms a face of $\hat{\Delta}_{ss}(n)$. Hence $\bigcup_i\pi(\sigma_i)$ is in the star of $\sigma_1$, with the image of $|\tau|$ contained in $|\bigcup_i\pi(\sigma_i)|$, so that $\pi(|\tau|)\subset|\bigcup_i\pi(\sigma_i)|\subset C(\tau)$, as desired. 

\section{Homeomorphism Types} 

In addition to homotopy types, we may also consider the homeomorphism types of $\hat{\Delta}_{ss}(n)$ and $\hat{\Delta}_{ws}(n)$. Results by Leclerc and Zelevinsky \cite[Theorem 1.6]{Leclerc} and by Danilov, Karzanov, and Koshevoy \cite[Theorem B]{DKK} on the maximal sizes of weakly separated and strongly separated collections, respectively, imply that both $\hat{\Delta}_{ss}(n)$ and $\hat{\Delta}_{ws}(n)$ are pure of dimension ${n-1 \choose 2}-1$. Based on these results and on Figures 1 and 2, one may question whether or not  $\hat{\Delta}_{ss}(n)\cong S^{n-3}\times B^{{n-2\choose 2}}$ and $\hat{\Delta}_{ws}(n)\cong B^{{n-1\choose 2}-1}$. In fact, neither of these is true. This is due to the fact that the boundary of $\hat{\Delta}_{ss}(5)$ (expected to be homeomorphic to $S^2\times S^2$) was found to have nontrivial reduced homology groups 
$$\tilde{H}_2(\partial\hat{\Delta}_{ss}(5))\cong\mathbf{Z}, \quad \tilde{H}_3(\partial\hat{\Delta}_{ss}(5))\cong\mathbf{Z}^9, \quad \tilde{H}_4(\partial\hat{\Delta}_{ss}(5))\cong\mathbf{Z},$$
and that the boundary of $\hat{\Delta}_{ws}(5)$ (expected to be homeomorphic to $S^4$) was found to have nontrivial reduced homology groups
$$\tilde{H}_2(\partial\hat{\Delta}_{ws}(5))\cong\mathbf{Z}, \quad \tilde{H}_4(\partial\hat{\Delta}_{ws}(5))\cong\mathbf{Z}.$$

The task of formulating a new conjecture on the homeomorphism types of $\hat{\Delta}_{ss}(n)$ and $\hat{\Delta}_{ws}(n)$ is made more difficult by virtue of the fact that $\partial\hat{\Delta}_{ss}(n)$ and $\partial\hat{\Delta}_{ws}(n)$ are not, in general, manifolds. Using the software package {\tt polymake}, we were able to determine that neither are manifolds even in the $n=5$ case. If they were, then the link of every $d$-dimensional face would have the homology of a $(3-d$)-sphere; however, in the case of $\partial\hat{\Delta}_{ws}(5)$, the links of the vertices $15$ and $234$ were found to have reduced homology groups $\tilde{H}_1$ and $\tilde{H}_3$ isomorphic to $\mathbf{Z}$. Further computations showed that $\lk(\{15, 234\})$, the only link of an edge in the boundary without the homology of a $2$-sphere, is the disjoint union of two boundaries of octahedra, as pictured in Figure $3$. In the case of $\partial\hat{\Delta}_{ss}(5)$, the link of the face $\{2, 23, 234\}$ within $\partial\hat{\Delta}_{ss}(5)$, which is expected to have the homology of a $1$-sphere, is homeomorphic to two disjoint circles, as pictured in Figure $4$. 

\begin{figure}
\begin{tikzpicture}
[inner sep=0pt, minimum size=7mm]

\node (134) at (-4.1,-2) [circle, draw, fill=white] {134};
\node (245) at (3.9,-2) [circle, draw, fill=white] {245};

\node (34) at (-6.3,-0.25) [circle, draw, fill=white] {34};
\node (35) at (-2.3,-0.25) [circle, draw, fill=white] {35};
\node (14) at (1.7,-0.25) [circle, draw, fill=white] {14};
\node (23) at (5.7,-0.25) [circle, draw, fill=white] {23};

\node (124) at (-6,1) [circle, draw, fill=white] {124};
\node (125) at (-2,1) [circle, draw, fill=white] {125};
\node (145) at (2,1) [circle, draw, fill=white] {145};
\node (235) at (6,1) [circle, draw, fill=white] {235};

\node (25) at (-4.1,2.5) [circle, draw, fill=white] {25};
\node (13) at (3.9,2.5) [circle, draw, fill=white] {13};

\draw[-] (25) to (124);
\draw[-] (25) to (125);
\draw[-] (25) to (34);
\draw[-] (25) to (35);
\draw[-] (134) to (34);
\draw[-] (134) to (35);
\draw[style=dashed] (134) to (124);
\draw[style=dashed] (134) to (125);
\draw[-] (34) to (35);
\draw[-] (34) to (124);
\draw[-] (35) to (125);
\draw[style=dashed] (124) to (125);

\draw[-] (13) to (145);
\draw[-] (13) to (235);
\draw[-] (13) to (14);
\draw[-] (13) to (23);
\draw[-] (245) to (14);
\draw[-] (245) to (23);
\draw[style=dashed] (245) to (145);
\draw[style=dashed] (245) to (235);
\draw[-] (14) to (23);
\draw[-] (14) to (145);
\draw[-] (23) to (235);
\draw[style=dashed] (145) to (235);

\begin{pgfonlayer}{background}
\fill[gray!30] (-4.1,-2) to (-6.3,-0.25) to (-2.3,-0.25);
\fill[gray!30] (-6.3,-0.25) to (-2.3,-0.25) to (-4.1,2.5);
\fill[gray!30] (-6.3,-0.25) to (-6,1) to (-4.1,2.5);
\fill[gray!30] (-2.3,-0.25) to (-2,1) to (-4.1,2.5);

\fill[gray!30] (3.9,-2) to (1.7,-0.25) to (5.7,-0.25);
\fill[gray!30] (3.9,2.5) to (1.7,-0.25) to (5.7,-0.25);
\fill[gray!30] (1.7,-0.25) to (2,1) to (3.9,2.5);
\fill[gray!30] (5.7,-0.25) to (6,1) to (3.9,2.5);
\end{pgfonlayer}

\end{tikzpicture}
\caption{The link $\lk(\{15, 234\})$ within $\partial\hat{\Delta}_{ws}(5)$.}
\end{figure}
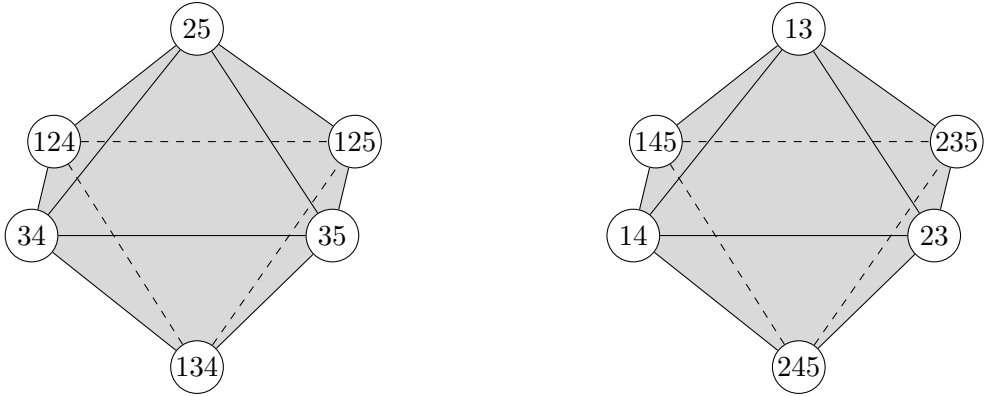

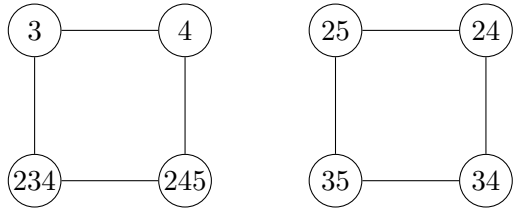
\begin{figure}
\begin{tikzpicture}
[inner sep=0pt, minimum size=7mm]

\node (3) at (-3,1) [circle, draw] {3};
\node (4) at (-1,1) [circle, draw] {4};
\node (234) at (-3,-1) [circle, draw] {234};
\node (245) at (-1,-1) [circle, draw] {245};

\node (25) at (1,1) [circle, draw] {25};
\node (24) at (3,1) [circle, draw] {24};
\node (35) at (1,-1) [circle, draw] {35};
\node (34) at (3,-1) [circle, draw] {34};

\draw[-] (3) to (4);
\draw[-] (3) to (234);
\draw[-] (4) to (245);
\draw[-] (234) to (245);

\draw[-] (25) to (24);
\draw[-] (25) to (35);
\draw[-] (35) to (34);
\draw[-] (24) to (34);

\end{tikzpicture}
\caption{The link $\lk(\{2, 23, 234\})$ within $\partial\hat{\Delta}_{ss}(5)$.}
\end{figure}

We note that the homology of $\partial\hat{\Delta}_{ws}(5)$ in particular indicates that we may have a homeomorphism $\partial\hat{\Delta}_{ws}(5)\cong S^2\vee S^4$, which is also not a manifold. Hence there may be some hope that, in general, $\partial\hat{\Delta}_{ws}(n)$ is a wedge of spheres. Computing the homology of $\partial\hat{\Delta}_{ws}(6)$ already becomes too time-consuming, however, and therefore no conjecture on the homeomorphism type of $\partial\hat{\Delta}_{ws}(n)$ has been posed.

\section*{Acknowledgments}

This research was conducted at the 2011 summer REU (Research Experience for Undergraduates) program at the University of Minnesota, Twin Cities, and was supported by NSF grants DMS-1001933 and DMS-1067183. The program was directed by Profs. Gregg Musiker, Pavlo Pylyavskyy, and Vic Reiner, whom the authors thank for their leadership and support. The authors would like to express their particular gratitude to Prof. Reiner for introducing them to this problem and for his indispensable guidance throughout the research process.

\end{document}